\documentclass[10pt,centertags]{amsart}

\usepackage{amsmath}
\usepackage{amsthm}
\usepackage{amssymb}
\usepackage{mathrsfs} % for calligraphic letters
\usepackage[pdftex, a4paper]{hyperref} %for hyperlinks
\usepackage{exscale}
\usepackage{verbatim}
\usepackage[arrow, matrix, curve]{xy}

\numberwithin{equation}{section}

\allowdisplaybreaks

\theoremstyle{definition}
	\newtheorem{definition}{Definition} %[subsection]
	\newtheorem*{definition*}{Definition}
	
	\numberwithin{definition}{section}
\theoremstyle{plain}
	\newtheorem{lemma}[definition]{Lemma}
	\newtheorem{proposition}[definition]{Proposition}
	\newtheorem{theorem}[definition]{Theorem}
	\newtheorem{problem*}[definition]{Problem}
	\newtheorem*{theorem*}{Theorem}
	\newtheorem{corollary}[definition]{Corollary}

	\newtheorem{conjecture}[definition]{Conjecture}
	
\theoremstyle{remark}

	\newtheorem{remark}[definition]{Remark}

\renewcommand{\bf}{\textbf}

\renewcommand{\phi}{\varphi}

\newcommand{\C}{\mathbb{C}}
\newcommand{\Z}{\mathbb{Z}}
\newcommand{\N}{\mathbb{N}}
\newcommand{\R}{\mathbb{R}}

\newcommand{\xto}{\xrightarrow}

\newcommand{\cprime}{$'$}

\renewcommand{\L}{\mathfrak}

\renewcommand{\L}{\mathfrak}

\begin{document}
\bibliographystyle{plain}

\title[Orders on Lie Groups]{Invariant Orders on Hermitian Lie Groups}
\author{Gabi Ben Simon and Tobias Hartnick}
\date{\today }
\begin{abstract} We study three natural bi-invariant partial orders on a certain covering group of the automorphism group of a bounded symmetric domain of tube type; these orderings are defined using the geometry of the Shilov boundary, Lie semigroup theory and quasimorphisms respectively. Our main result shows that these orders are related by two inclusion relations. In the case of ${\rm SL_2}(\R)$ we can show that they coincide. We also prove a related coincidence of orders for the universal covering of the group of homeomorphisms of the circle.
\end{abstract}
\maketitle

\section{Introduction}

Consider the group $G_0 := PU(1,1)$ of biholomorphic automorphisms of the Poincar\'e disc and its universal covering group $G$. There are (at least) three different ways to define a bi-invariant partial order on $G$:

Firstly, the action of $G_0$ on the disc by rational linear transformations extends to the boundary $S^1$, and this boundary action lifts to an action of $G$ on the real line. Then the natural order on $\R$ induces a bi-invariant partial order on $G$ via
\[g \geq h :\Leftrightarrow \forall x \in \R: g.x \geq h.x.\]
We refer to this order as the \emph{geometric order} on $G$.\\

In order to state the definitions of the other two orders on $G$, we remark that any bi-invariant partial order $\leq$ on $G$ is uniquely determined by its \emph{order semigroup} $G^+ := \{g \in G\,|\, g \geq e\}$, hence we can define the other two orders by giving their order semigroups.\\

The best studied class of subsemigroups of Lie groups is the class of Lie semigroups \cite{HHL, Neeb, HiNe}. Invariant orders on simple Lie groups arising from Lie semigroups have been classified by Ol{\cprime}shanski{\u\i} \cite{Olshanski}, following the pioneering work of Vinberg \cite{Vinberg}. In the present case their classification is particularly simple: The Lie algebra $\L g$ of $G$ admits a unique pair of ${\rm Ad}$-invariant closed pointed convex cones $\pm C$. One of these two cones, say $C$, exponentiates into the order semigroup of the geometric order. We refer to it as the \emph{positive} cone. Now denote by $G^+$ the closure of the semigroup generated by $\exp(C)$. Then $G^+$ is the order semigroup of a bi-invariant partial order on $G$ and at the same time a Lie semigroup. We thus refer to the associated order as the \emph{Lie semigroup order} on $G$. Up to inversion, it is the unique continuous order on $G$ in the sense of \cite{Neeb}. By construction, this order is refined by the geometric order. (Here an invariant order $\leq$ is said to \emph{refine} another order $\preceq$ if the order semigroup of $\leq$ contains the order semigroup of $\preceq$.)\\

Finally, we propose a third way to introduce a bi-invariant partial order on $G$, which is not so classical as the two construction above, but generalizes nicely to other types of groups. Given any family $\mathcal F$ of real valued continuous functions  on $G$ we obtain a closed semigroup $S_{\mathcal F}$ of $G$ by setting
\[S_{\mathcal F} :=\{g \in G\,|\, \forall f \in \mathcal F\, \forall h \in G: f(hg) \geq f(h) \}.\]
As pointed out in \cite[Prop. 1.19]{HiNe}, every closed submonoid of $G$ is of the form $S_{\mathcal F}$. However, in general the set $\mathcal F$ will be large, and it can be a hard problem to find a reasonably small set $\mathcal F$ for a given closed monoid $S$. (For example, it is highly non-trivial to show that $\mathcal F$ can always be chosen to consist of analytic functions, see \cite[Thm. 1.29]{HiNe}.) From this point of view, semigroups $S_{\{f\}}$ associated with a single continuous function are rather special.\\

Now, on $G$ there is a distinguished continuous function $T$ called the \emph{translation number}, which also arises from the action of $G_0$ on the circle. To give a precise definition, we first recall that Poincar\'e's \emph{rotation number} on the group of orientation preserving homeomorphisms on the circle \cite{Po1, Po2} pulls back to a function $R: G_0 \to \R/\Z$ via the $G_0$-action on the circle discussed above; then $T$ is the unique continuous lift of $R$ with $T(e) = 0$. In terms of the $G$-action on $\R$, this function can be expressed as
\[T(g) = \lim_{n\to\infty}\frac{g^n.x-x}{n}\quad(g \in G, x \in \R).\]
The semigroup $S_{\{T\}}$ associated with $T$ turns out to define a bi-invariant partial order; moreover, it is a maximal (proper) conjugation-invariant subsemigroup of $G$. We thus denote it by $G^+_{\max}$ and refer to the associated order as the \emph{maximal order} on $G$. This maximal order is easily seen to refine both the geometric and the Lie semigroup order. One may now wonder, whether any of the above refinements is proper. In fact, they are not:
\begin{theorem}\label{MainThmRk1} The geometric order, the Lie semigroup order and the maximal order on $G$ coincide. In particular, the Lie semigroup $G^+$ is a maximal conjugation-invariant subsemigroup of $G$ and can be described by a single continuous function.
\end{theorem}
Theorem \ref{MainThmRk1} will be proved in Proposition \ref{SL2case}. Note that the function $T$ defining our semigroup $G$ is only continuous, not $C^1$, and in particularly not analytic.\\

The three orders referred to in Theorem \ref{MainThmRk1} can be defined for much more general groups than the universal cover of $SL_2(\R)$. For example, consider the (non-compact) symplectic group $G_0 := {\rm Sp}(2n, \R)$ and its universal covering $G$. Then the Lie algebra $\L g$ of $G$ still contains a unique pair of invariant cones $\pm C$ and the \emph{Lie semigroup order} can still be defined \cite{Olshanski}. The role of the translation number in the $SL_2(\R)$-example is now taken by the Maslov quasimorphism $\mu_{G}$ on $G$, i.e. the semigroup $S_{\{\mu_G\}}$ is a maximal conjugation-invariant subsemigroup of $G$, which gives rise to a refinement of the Lie semigroup order called the \emph{maximal order} on $G$ (see Section \ref{MaxQM} below). It is also possible to define a \emph{geometric order} in this setup: The action of $G_0$ on the corresponding Lagrangian Grassmannian $\Lambda(\R^{2n})$ lifts to an action of $G$ on the universal covering of $\Lambda(\R^{2n})$. This universal covering admits a natural ordering modelled on the cone of positive-definite quadratic form, which in turn induces a geometric ordering on $G$ \cite{ClercKoufany, BSH2}.\\

Basically there is a similar picture for any simple Hermitian Lie group, but there are some caveats. Let us assume that $G_0$ is the biholomorphism group of an irreducible bounded symmetric domain $\mathcal D$ and denote by $\widetilde{G_0}$ its universal covering. Unlike the case of the symplectic groups, the fundamental group of $G_0$ may contain torsion elements, and the most natural place to compare orders is \emph{not} $\widetilde{G_0}$ but the quotient $G$ of $\widetilde{G_0}$ by the torsion part of $\pi_1(G_0)$. On this group there exists always a unique (up to multiples) aperiodic homogeneous quasimorphism\footnote{See Section \ref{MaxQM} for precise definitions.} $\mu_G$ generalizing the translation number, respectively the homogeneization of the Maslov index, and the \emph{maximal order} can be defined as before. Now, as far as the geometric order is concerned, it will exist if and only if the bounded symmetric domain $\mathcal D$ in question is of tube type. In this case (and only in this case) there exists a unique (up to inversion) causal structure on the Shilov boundary $\check S$ of $\mathcal D$ \cite{Kaneyuki}, which gives rise to a partial order on the universal cover $\check R$ of $\check S$. Then the \emph{geometric order} on $G$ is the one induced from the action of $G$ on $\check R$. (It is easy to see that $G$ acts effectively on $\check R$ so that this is well-defined.) The geometric order was first explored systematically in \cite{ClercKoufany}, see also \cite{BSH2} and \cite {Konstantinov}. (There is a small issue about the precise definition; we will always work with closed orders here, while \cite{ClercKoufany} does not. This will be discussed in more details in Section \ref{OtherOrders} below.) As first pointed out in \cite{Vinberg}, the group $G$ also carries a bi-invariant partial order induced by a Lie semigroup. However, as proved in \cite{Olshanski}, if $G$ is not locally isomorphic to a symplectic group, then there is in fact a continuum of Lie semigroup orders on $G$. The order that we are interested in here, is the \emph{maximal} Lie semigroup order on $G$. By \cite{Neeb}, this is at the same time the \emph{maximal continuous} bi-invariant partial order on $G$, and we prefer the latter term (since it can be defined for more general topological groups). Now we have the following result:
\begin{theorem}
The above bi-invariant partial orders are related as follows:
The maximal order is a refinement of the maximal continuous order. If $G$ is of tube type, then the geometric order refines the maximal continuous order and is refined by the maximal order. In symbols,
\[G^+_{cont} \subset G^+_{geom}\subset G^+_{\max}.\]
\end{theorem}
The first statement follows immediately from work in \cite{BSH}. The fact that the geometric order refines the maximal continuous one follows from work of Konstantinov \cite{Konstantinov}. The last statement is a consequence of the results in \cite{BSH2}, which in turn refine results of Clerc and Koufany from \cite{ClercKoufany}. We will provide details in Section \ref{OtherOrders} below. In view of Theorem \ref{MainThmRk1} it is tempting to conjecture:
\begin{conjecture}\label{MainConjecture}
The maximal and maximal continuous bi-invariant order coincide. In particular they both coincide with the geometric order in the tube type case.
\end{conjecture}
We provide some evidene for the conjecture in Proposition \ref{Evidence} below by showing that both $G^+_{cont}$ and $G^+_{\max}$ share the same Lie wedge and that their interiors share the same divisible hull; moreover, $G^+_{cont}$ and $G^+_{\max}$ are both closed and path-connected. Unfortunately, we are not able to decide whether this is enough to deduce Conjecture \ref{MainConjecture}. Our proof in the $SL_{2}(\mathbb{R})$-case is based on a criterion, which guarantees abstract maximality of Lie semigroups under a certain hypothesis (open dominants), see Corollary \ref{RecognitionLie}. This hypothesis will not be satisfied for maximal pointed invariant Lie subsemigroups of more general Hermitian Lie groups, so our proof does not carry over directly. It seems likely that one would need more general criteria guaranteeing abstract maximality of Lie semigroups. To the best of our knowledge this question has, unfortunately, not been treated systematically in the literature yet.\\

The construction of the order $G^+_{\max}$ can be carried out for any aperiodic quasimorphism on a topological group $G$. As far as finite-dimensional Lie groups are concerned, the number of examples of such quasimorphisms is limited; in fact, it is folklore that they can be classified, and we state the classification explicitly in Theorem \ref{Classification}. All of them live on reductive Hermitian Lie groups and arise essentially from the examples covered by Conjecture \ref{MainConjecture} (although it should be said that the passage from the simple to the reductive case is non-trivial on the level of orders). For genuinely different examples one has to turn to either non-connected groups (in particular, discrete groups of negative curvature, see e.g. \cite{EpFu, BeFu} and the references therein) or infinite-dimensional Lie groups (where quasimorphisms arise  for example from problems in symplectic and contact geometry, see e.g. \cite{Polt, EntP, GabiNonlinear, Py, Shel}).\\

While the focus of the present article is mainly on the case of finite-dimensional Lie groups, we found it worthwhile to develop the basic theory of bi-invariant orders associated to quasimorphisms in a generality appropriate for the treatment of these discrete or infinite-dimensional examples, whenever this was possible without too many additional efforts. To provide an application of the abstract theory beyond finite-dimensional Lie groups, we compute the maximal invariant order associated with Poincar\'e's translation number on the universal covering of the group of orientation-preserving homeomorphisms on the circle.\\

This article is organized as follows: In Section \ref{AbstractNonsense} we develop some fundamentals of the abstract theory of semigroups associated with quasimorphisms. The material here applies to general topological groups and might be of independent interest. In particular, we provide in Theorem \ref{Recognition} a recognition criterion for semigroups associated with quasimorphisms. We then specialize in Section \ref{SecLie} to the case of Lie groups and apply the recognition criterion in the case of $SL_2(\R)$. The final section describes how the geometric orders on tube type groups fit into the picture.\\

\bf{Acknowledgements:} The definition of the maximal order associated with a quasimorphism arose from discussions with Marc Burger and Danny Calegari. Joachim Hilgert, Karl Heinrich Hofmann, Karl-Hermann Neeb and Grigori Ol{\cprime}shanski{\u\i} kindly answered various of our questions about Lie semigroups. We thank all of them and also the anonymous referee, who was of great help in bringing the article into its final form. The authors also acknowledge the hospitality of IHES, Bures-sur-Yvette, and Hausdorff Institute, Bonn. The first named author is grateful to Departement Mathematik of ETH Z\"urich  and in particular Dietmar Salamon and Paul Biran for the support during this academic year. The second named author was supported by SNF grants PP002-102765 and 200021-127016.

\section{Aperiodic homogeneous quasimorphism and their maximal orders}\label{AbstractNonsense}

\subsection{The maximal order of an aperiodic quasimorphism}\label{MaxQM}
Let $G$ be a topological group. A continuous function $f: G \to \R$ is called a \emph{quasimorphism} if the function
\[\partial f: G^2 \to \R, \quad \partial f(g,h) = f(gh)-f(g) -f(h)\]
is bounded. In this case the \emph{defect} of $f$ is defined as the real number
\[D(f) := \sup_{g,h \in G} |\partial f(g,h)|.\]
A quasimorphism $f: G \to \R$ is called \emph{homogeneous} if
\[f(g^n) = nf(g) \quad (g \in G, n \in \mathbb Z).\]
Given a group $G$ we denote by $\mathcal{HQM}(G)$ the space of homogeneous quasimorphisms on $G$. Let us collect some basic facts concerning quasimorphisms, which are easily available from the standard literature on the subject (see e.g. \cite{scl}): Every quasimorphism is at bounded distance from a unique homogeneous one called its \emph{homogeneization}; in particular every bounded homogeneous quasimorphism is necessarily trivial and two quasimorphisms are at bounded distance if and only if their homogenizations coincide. Homogeneous quasimorphism share many properties with $\R$-valued homomorphisms, in particular they are automatically invariant under conjugation.
 If a group $G$ is amenable, then every homogeneous quasimorphism on $G$ is in fact a homomorphism. This applies in particular to all compact, all abelian and, more generally, all virtually solvable groups. Constructing homogeneous quasimorphisms which are not homomorphisms is more complicated; we will discuss various examples below.

 \begin{definition}
 A normal subgroup $H \lhd G$ is called a \emph{period group} for $f \in \mathcal{HQM}(G)$ if $f|_H= 0$; the quasimorphism $f$ is called \emph{aperiodic} if it does not admit a non-trivial period group. A maximal (with respect to inclusion) period subgroup of $f$ is called a \emph{kernel} of $f$.
 \end{definition}
  Given two period groups $H_1, H_2$ for a quasimorphism $f$, their product $H_1H_2$ is again a period group. Indeed, the product is a normal subgroup of $G$, since both $H_j$ are, and $f$ is bounded (by its defect) on the group $H_1H_2$, hence trivial. This implies that every homogeneous quasimorphism $f$ admits a unique kernel, which we denote by $\ker(f)$. We then obtain the following canonical factorization:
\begin{proposition}\label{Factorization1} Let $G$ be topological group and $f \in \mathcal{HQM}(G)$ be a homogeneous quasimorphism. Then there exists a unique topological group $\hat G$ and a unique factorization of $f$ as
\[f= f_0 \circ p: G \xrightarrow{p} \hat G \xrightarrow{f_0} \R,\]
where $p$ is an open surjective continuous homomorphism and $f_0$ is an aperiodic homogeneous quasimorphism.
\end{proposition}
In view of the proposition we will mainly focus on aperiodic quasimorphisms in the sequel.\\

We now recall some results from \cite{BSH}, which motivated the present investigation. The main upshot of that article is that quasimorphisms are closely related to bi-invariant partial orders on $G$. Here a partial order $\leq$ on $G$ is called \emph{bi-invariant},  if for all $g,h,k \in G$ the implication
\[g \leq h \Rightarrow (kg \leq kh \wedge gk \leq hk)\]
holds. Equivalently, its order semigroup
\[G^+_{\leq} := G^+ := \{g \in G\,|\, g \geq e\}\]
is a conjugation-invariant submonoid subject to the condition $G^+ \cap (G^+)^{-1} = \{e\}$. (A set with the latter property will be called \emph{pointed}.)  Bi-invariant orders have the following important multiplicativity property:
\begin{lemma}\label{MultiplyInequalities} For every bi-invariant partial order $\preceq$ on $G$ we have
\[g_1 \preceq g_2, \; h_1 \preceq h_2 \Rightarrow g_1h_1 \preceq g_2h_2.\]
\end{lemma}
\begin{proof} If $G^+$ denotes the associated order semigroup, then $h_2h_1^{-1} \in G^+$ and thus $g_2h_2h_1^{-1} g_2^{-1} \in G^+$ by conjugation invariance. Now using the semigroup property of $G^+$ we find $(g_2h_2h_1^{-1} g_2^{-1})(g_2g_1^{-1}) \in
G^+$.
\end{proof}
We now recall from \cite{BSH} that every nonzero quasimorphism $f \in \mathcal{HQM}(G)$ defines a bi-invariant partial order $\leq$ with order semigroup $G^+_f$ given by
\[G^+_f := \{g \in G\,|\, f(g) > D(f)\} \cup \{e\},\]
and that $f$ can be recovered from $G^+$ up to a positive multiple via the notion of relative growth as introduced in \cite{EP}. In fact, $f$ can be recovered up to a positive multiple from any bi-invariant partial order $\leq$ on $G$, whose order semigroup $G^+$ satisfying the following more general condition \cite[Prop. 3.3]{BSH}:
\begin{eqnarray}\label{sw}\exists C_1, C_2 \in \R: \, \{g \in G\,|\, f(g) \geq C_1\} \subset G^+_{\leq} \subset \{g \in G \,|\, f(g) \geq C_2\}.\end{eqnarray}
This motivates a systematic investigation of the collection $\mathcal IPO_f(G)$ of all orders satisfying \eqref{sw}, or equivalently, the collection $\mathcal M_f(G)$ of associated order semigroups. We refer to condition \eqref{sw} as the \emph{sandwich condition}, and say that elements of  $\mathcal IPO_f(G)$ or $\mathcal M_f(G)$ are \emph{sandwiched} by $f$. The present formulation of the sandwich condition is very symmetric and stresses the idea of obtaining sandwiched orders by varying the constant in $G^+_f$. However, for practical purposes, the following asymmetric version is more useful:
\begin{lemma}[{\cite[Lemma 3.2]{BSH}}]\label{NoUpperBound}Let $G$ be group, $f: G \to \R$ a non-trivial homogeneous quasimorphism and $\leq$ be a bi-invariant partial oder on $G$ with order semigroup $G^+$. If there exists $C_1>0$ with
\[\{g \in G\,|\, f(g) \geq C_1\} \subset G^{+},\]
then $\leq$ is sandwiched by $f$ and \eqref{sw} is satisfied with $C_2 := 0$.
\end{lemma}
From now on we assume that $f$ is aperiodic and $G \neq \{e\}$. This implies in particular $f \neq 0$, whence $\mathcal IPO_f(G)$ and $\mathcal M_f(G)$ are infinite. We are interested in finding a canonical representative for $\mathcal IPO_f(G)$. To this end we observe that $\mathcal M_f(G)$ is partially ordered by inclusion; this induces a partial order on $\mathcal IPO_f(G)$. Under the present assumptions there is always a unique maximal element:
\begin{proposition}\label{CalegariOrder} Suppose $f \in \mathcal{HQM}(G)$ is aperiodic. Then there exists a unique maximal element $\leq \in \mathcal IPO_f(G)$. Its associated order semigroup is explicitly given by
\begin{eqnarray}\label{GMax}G^+_{\max} = \{g \in G \,|\, \forall h \in G: f(gh) \geq f(h)\}.\end{eqnarray}
\end{proposition}
\begin{proof} We may assume $G \neq \{e\}$. We then have to show that $G^+_{\max}$ is the unique maximal element of
$\mathcal M_f(G)$ with respect to inclusion. For this let $G^+ \in \mathcal M_f(G)$. By Lemma \ref{NoUpperBound} we may assume that \eqref{sw} is satisfied with $C_2 := 0$. Now let $g \in G^+$ and $h \in G$ arbitrary. Then $gh \geq h$, whence $(gh)^n \geq h^n$ for all $n \in \mathbb N$ by Lemma \ref{MultiplyInequalities}. Since
\[G^+ \subset \{g \in G\,|\, f(g) \geq 0\}\]
we deduce that
\[f((gh)^nh^{-n}) \geq 0,\]
whence
\[f(gh) - f(h) + \frac{D(f)}{n} \geq 0.\]
For $n \to \infty$ we obtain $g \in G^+_{\max}$ and thus $G^+ \subset G^+_{\max}$. It remains to show that $G^+_{\max} \in \mathcal M_f(G)$. If $g_1, g_2 \in G^+_{\max}$ then for all $h \in G$ we have
\[f(g_1g_2h) \geq f(g_2h) \geq f(h), \]
whence $g_1g_2 \in G^+_{\max}$, while $e \in G^+_{\max}$ is obvious. Finally, we recall that every homogeneous quasimorphism is conjugation-invariant, hence for all $g \in  G^+_{\max}$ and $h, k \in G$ we obtain
\[f(kgk^{-1}h) = f(gk^{-1}hk) \geq f(k^{-1}hk) = f(h).\]
We now claim that $G^+_{\max}$ is pointed. Otherwise, $H := G^+_{\max} \cap (G^+_{\max})^{-1}$ is a non-trivial normal subgroup of $G$. Moreover, for $g \in H$ we have
\[g \in G^+_{\max} \Rightarrow f(g) = f(ge) \geq f(e) = 0\]
and similarly
\[g \in (G^+_{\max})^{-1} \Rightarrow g^{-1} \in  G^+_{\max} \Rightarrow -f(g) = f(g^{-1}) = f(g^{-1}e) \geq f(e) = 0,\]
whence $f(g) = 0$. Thus $f|_H \equiv 0$, contradicting the aperiodicity of $f$. Finally, let us show that $f$ sandwiches $G^+_{\max}$: If $f(g) \geq D(f)$, then for all $h \in G$ we have
\[f(gh) \geq f(g) + f(h) - D(f) \geq f(h)\]
showing that $g \in G^+_{\max}$. By Lemma \ref{NoUpperBound} this suffices to finish the proof.
\end{proof}
\begin{remark}
The definition of $G^+_{\max}$ looks asymmetric on the first sight: One could as well ask for the condition $f(hg) \geq f(h)$ for all $h \in G$. However, due to the conjugation-invariance of
$f$ we have
\[\forall h \in G: f(gh) \geq f(h) \Leftrightarrow  \forall h \in G: f(hg) \geq f(h).\]
Therefore it is enough to demand one of the two conditions here.
\end{remark}
We refer to the order defined in Proposition \ref{CalegariOrder} as the \emph{maximal order} associated to the aperiodic quasimorphism $f$. It is easy to deduce from either the explicit formula or the abstract maximality property that maximal orders are closed. Here maximality has to be understood in the following sense:
\begin{proposition}\label{PointedMaximality} Let $\leq$ be the maximal order associated with some aperiodic nonzero $f \in \mathcal{HQM}(G)$. Then the associated order semigroup $G^{+}_{\max}$ is a maximal pointed conjugation-invariant subsemigroup of $G$, i.e. there is no pointed conjugation-invariant semigroup $S$ with
\[G^+_{\max} \subsetneq S \subsetneq G.\]
\end{proposition}
\begin{proof} In view of Lemma \ref{NoUpperBound} every pointed semigroup $S$ as above is contained in $\mathcal M_f(G)$, hence the result follows.
\end{proof}
In many special cases, such as simple Lie groups, maximal orders satisfy a much stronger maximality condition; we will return to this question below. In any case, we can now formulate the main problem to be discussed in this article:
\begin{problem*}\label{MainProblem}
Given a topological group $G$ and an aperiodic homogeneous quasimorphism $f$ on $G$, determine the associated maximal order $G^+_{\max}$ explicitly.
\end{problem*}
We will discuss this question for finite-dimensional Lie groups below, but before we can do so we need to introduce some general tools.

\subsection{Maximal dominant sets}
We recall that the \emph{dominant set} of a bi-invariant partial order $\leq$ on a topological group $G$ is the subsemigroup of the order semigroup $G^+$ given by the formula
\[G^{++} := \{h \in G^+ \setminus\{e\}\,|\, \forall g \in G\,\exists n \in \N:h^n \geq g\}.\]
This notion was introduced in \cite{EP}. If $\leq$ is sandwiched by a non-zero quasimorphism $f$ on $G$, then this is always non-empty.
For connected groups $G$ the dominant set is related to the interior of the order semigroup:
\begin{lemma}\label{TrivialTopologicalStuff}
Let $G$ be a connected topological group, $f \in \mathcal{HQM}(G) \setminus\{0\}$ and $\leq \in \mathcal{IPO}_f(G)$. Denote by $G^+$ respectively $G^{++}$ the order semigroup and dominant set of some bi-invariant partial order on $G$. Then the following hold:
\begin{itemize}
\item[(i)] ${\rm Int}(G^+) \subset G^{++}$.
\item[(ii)] There exists a constant $C>0$ such that for all $g \in G$
\[f(g) > C \Rightarrow g \in {\rm Int}(G^+).\]
\end{itemize}
\end{lemma}
\begin{proof} (i) Let $g \in {\rm Int}(G^+)$ and let $U \subset G$ be
open with $g \in U \subset G^+$. Then the semigroup generated by $G^+$ and $g^{-1}$ contains an open identity neighborhood, hence coincides with $G$, since $G$ is connected. Using the conjugation-invariance of $G^+$ this implies that every $h \in G$ may be written as $h = g^{-n}h_+$ for some $h_+\in G^+$ and $n \in \mathbb N$. We then have $g^nh = h_+ \geq e$; since $h\in G$ was arbitary, this implies that $g$ is dominant. (We learned this argument from K.-H. Hofmann.) (ii) In view of the sandwich condition it suffices to show that $f|_{\partial G^+}$ is bounded. For
this we argue as follows: Let $C_0$ be a sandwich constant so that $\{g \in G\,|\, f(g) \geq C_0\} \subset G^+$. We claim that $f|_{\partial G^+}$ is bounded by $C := 2C_0$.  Indeed,
suppose $g \in \partial G^+$ with $f(g) > 2C_0$. We then find a sequence of elements
$g_n \in G \setminus G^+$ with $g_n \to g$. In particular, $f(g_m) \geq C_0$ for $m$ sufficiently large, whence $g_m \in G^+$, a contradiction.
\end{proof}
We denote by $\mathcal D_f(G)$ the collection of dominant sets of the elements of $\mathcal M_f(G)$. We aim to describe $\mathcal D_f(G)$ in more intrinsic terms. The key observation allowing for such a description is as follows:
\begin{lemma}\label{DominantsFromOSG} Let $G$ be a group, $f \in \mathcal{HQM}(G) \setminus \{0\}$ and $G^+ \in \mathcal M_f(G)$. Then the dominant set $G^{++}$ of $G^+$ is given by
\[G^{++} =  \{g \in G^+\,|\,f(g) > 0\}.\]
\end{lemma}
\begin{proof} Let $g \in G^{++}$. By Lemma \ref{NoUpperBound} there exists $C_1 > 0$ such that
\begin{eqnarray}\label{SandwichG1}\{g \in G\,|\, f(g) \geq C_1\} \subset G^+ \subset \{g \in G\,|\, f(g) \geq 0\},\end{eqnarray}
whence $f(g) \geq 0$. Assume $f(g) = 0$ for contradiction and observe that $f$ as a nonzero homogeneous quasimorphism is unbounded. Choose $\epsilon > 0$ and $h \in G$ with $f(h) \leq -D(f) - \epsilon$. Thus for all $n \in \mathbb Z$,
\[f(g^nh) \leq nf(g) + f(h) + D(f) \leq - \epsilon < 0,\]
whence $g^nh \not \in G^+$, i.e. $g^n \not \geq h^{-1}$, showing that $g$ is not dominant. This contradiction shows $f(g) > 0$, which yields the inclusion $\subseteq$. Conversely, if  $g \in G^+$ satisfies $f(g) > 0$
then given any $h \in G$ we find $m \in \mathbb N$ such that $mf(g) \geq f(h) + D(f)+C_1$, where $C_1$ is as in \eqref{SandwichG1}. With $m$ chosen in this way we have $f(g^mh^{-1}) \geq C_1$, hence $g^m \geq h$ by \eqref{SandwichG1}. This shows $g \in G^{++}$.
\end{proof}
From this we deduce:
\begin{corollary}\label{DominantsChara} Let $f \in \mathcal{HQM}(G) \setminus \{0\}$. Then a subset $G^{(++)}$ is contained in $\mathcal D_f(G)$ iff it satisfies the following three conditions:
\begin{itemize}
\item[(D1)] $G^{(++)}$ is a conjugation-invariant semigroup.
\item[(D2)] $f|_{G^{(++)}} > 0$.
\item[(D3)] $\exists C > 0\; \forall g \in G: f(g) \geq C \Rightarrow g \in G^{(++)}$
\end{itemize}
\end{corollary}
\begin{proof} First suppose, (D1)-(D3) are satisfied. We observe that (D2) implies in particular $G^{(++)} \cap (G^{(++)})^{-1} = \emptyset$. Combining this with (D1) we see that $G^+ := \{e\} \cup G^{(++)}$ is a pointed, conjugation-invariant monoid, hence the order semigroup of some bi-invariant partial order $\leq$, which by (D3) is sandwiched by $f$. Now Lemma
\ref{DominantsFromOSG} and (D2) imply that the set of dominants of $G^+$ is precisely $G^{(++)}$.\\
Conversely, supposse $G^{(++)}$ is the set of dominants for some partial order $\leq$ sandwiched by $f$ with order semigroup $G^+$. Then $G^+$ is conjugation invariant and since
$f$ is conjugation-invariant, (D1) follows from Lemma \ref{DominantsFromOSG}. The same lemma also yields (D2) immediately. Finally, (D3) follows
from the fact that $f$ sandwiches $\leq$ together with Lemma \ref{DominantsFromOSG}.
\end{proof}
Another consequence of Lemma \ref{DominantsFromOSG} is the following:
\begin{corollary}\label{DomIncl}If $G^+_1, G^+_2 \in \mathcal M_f(G)$ such that $G^+_1 \subset G^+_2$ then the associated sets of dominants $G^{++}_1, G^{++}_2 \in \mathcal D_f(G)$ satisfy
\[G^{++}_1 \subset G^{++}_2.\]
\end{corollary}
From this we deduce:
\begin{corollary} Let $f \in \mathcal{HQM}(G)$ be any nonzero homogeneous quasimorphism (not necessarily aperiodic). Then $\mathcal D_f(G)$ contains a unique maximal element $G^{++}_{\max}$. If $f$ is aperiodic, then $G^{++}_{\max}$ is the dominant set of $G^{+}_{\max}$.
\end{corollary}
\begin{proof} In the aperiodic case this follows from Proposition \ref{CalegariOrder} and Corollary \ref{DomIncl}. The general case is reduced to this by means of Proposition \ref{Factorization1}.
\end{proof}
From now on we will assume $G \neq \{e\}$, so that in particular every aperiodic quasimorphism is nonzero. For $f$ aperiodic we then have
\[G^{++}_{\max} = \{g \in G \,|\, \forall h \in G: f(gh) \geq f(h), f(g) > 0\}.\]
Conversely, we can recover $G^+_{\max}$ from its dominant set:
\begin{proposition}\label{MaximalFromDom} If $f \in \mathcal{HQM}(G)$ is nonzero and aperiodic then the unique maximal element $G^+_{\max} \in \mathcal M_f(G)$ is given by.
\begin{eqnarray*}G^+_{\max} = \{g \in G\,|\,
gG^{++}_{\max} \subset G^{++}_{\max} \}.\end{eqnarray*}
\end{proposition}
\begin{proof}
The main part of the proof consists of showing that \[G^+ :=  \{g \in G\,|\,
gG^{++}_{\max} \subset G^{++}_{\max} \}\] is pointed. For this we first show that
\begin{eqnarray}\label{NonNegativef} g \in G^+ \Rightarrow f(g) \geq 0.\end{eqnarray}
Indeed, if $g \in G^+$ then for all $m, n \in \mathbb N$ we have $g^{mn}h^n \in G^{++}_{\max}$, whence
\[f(g^{mn}h^n) > 0 \Rightarrow mf(g) + f(h) + \frac{D(f)}{n} > 0 \Rightarrow f(g) \geq \frac{-f(h)}{m} \Rightarrow f(g) \geq 0.\]
This proves \eqref{NonNegativef} and shows in particular that $f$ vanishes on $H := G^+ \cap
(G^+)^{-1}$. However, since the latter is a normal subgroup of $G$ and $f$ is aperiodic, we obtain $H = \{e\}$, whence $G^+$ is pointed. Now we can show that  $G^+  \in \mathcal M_f(G)$: Firstly, $G^+$ is a semigroup, since $g,h \in G^+$ implies $ghG^{++}_{\max} \subset gG^{++}_{\max} \subset G^{++}_{\max}$, and conjugation invariant, since for $g \in G^+, h \in G$ and $x \in
G^{++}_{\max}$ we have $y := h^{-1}xh \in G^{++}_{\max}$ and hence
\[hgh^{-1}x = hgh^{-1}hyh^{-1} = hgyh^{-1} \in hG^{++}_{\max}h^{-1} \subset G^{++}_{\max}.\]
Since obviously $e \in G^+$, the latter is a conjugation-invariant pointed monoid. It remains to prove that $f$ sandwiches $G^+$. Now $G^{++}_{\max} \subset G^+$ since
$G^{++}_{\max}$ is a semigroup. Since $G^{++}_{\max} \in \mathcal D_f(G)$ we find $C> 0$ such that
\[\{g \in G\,|\, f(g) \geq C\} \subset G^{++}_{\max} \subset G^+.\]
By Lemma \ref{NoUpperBound} we thus obtain $G^+ \in \mathcal M_f(G)$. Now let $S \in \mathcal M_f(G)$ and assume $G^+ \subseteq S$. Then the dominant set $S^{++}$ of $S$ contains $G^{++}_{\max}$ by Corollary \ref{DomIncl}, hence coincides with $G^{++}_{\max}$. Then $S \subset G^+$ since $S^{++} = G^{++}_{\max}$ is an ideal in $S$. This shows maximality of $G^+$, hence $G^+ = G^+_{\max}$.
\end{proof}

\subsection{A first example: the translation number} One of the most classical quasimorphism is Poincar\`e's \emph{translation number} $T$ \cite{Po1, Po2}. It is also one of the most important quasimorphisms, not only from the point of view of the structure theory of general quasimorphisms \cite{BSH2}, but also in terms of applications. For instance, it is  one of the key tools in the modern theory of group actions on the circle \cite{Ghys1, Ghys2}. We will now present the solution of Problem \ref{MainProblem} in the case of the translation number. Thereby we hope to illustrate the usefulness of the theory of dominants developed in the last section. We start by recalling the definition of the translation number: Let $H_0$ denote the group of orientation-preserving homeomorphisms of the circle (equipped with the compact open topology) and let $H$ be its universal covering. Explicitly,
\[H = \{f \in {\rm Homeo}(\R)\,|\, f \text{ monotone}, f(x+1) = f(x)+1\}.\]
Then for any $x \in \R$ we have
\[T(g) = \lim_{n \to \infty}\frac{g^n.x-x}{n}.\]
From this description it is not easy to decide whether $T$ is continuous. In fact it is continuous, as follows for example from the alternative description in \cite{Ghys1, Ghys2} as lift of the rotation number (whose continuity was known to Poincar\'e). Given this fact it is quite easy to see from the above description that $T$ is a homogeneous quasimorphism on $H$ (see again \cite{Ghys1, Ghys2} or alternatively \cite{scl, BSH2} for details). We claim that $T$ is in fact aperiodic. Indeed, the central extension $p: H \to H_0$ is non-trivial (as an element of $H^2(H_0; \Z)$ it is given by the Euler class), $H_0$ is simple (see e.g. \cite[Thm. 4.3]{Ghys2}), and $T$ restricts to an injective homomorphism on the kernel of $p$ (by the explicit formula). Thus aperiodicity of $T$ follows from the following general lemma:
\begin{lemma}
Let $G_0$ be a simple group, and $0 \to \mathcal Z \hookrightarrow G \xto{p} G_0 \to \{e\}$ be a non-trivial central extension. Then every homogeneous quasimorphism $f$ on $G$, which restricts to an injective homomorphism on $\mathcal Z$, is aperiodic.
\end{lemma}
\begin{proof} Assume that $N \lhd G$ is a period subgroup for $f$; then we have a short exact sequence
\[\{0\} \to N \cap \mathcal Z \to N \to p(N) \to \{e\}.\]
By assumption $N \cap \mathcal Z = \{0\}$. Since $G_0$ is simple we have either $p(N) = \{e\}$ or $p(N) = H_G$. In the second case we obtain a splitting of the extension defined by $p$. Since the extension was assumed to be non-trivial, this is impossible. Thus we are in the first case and $N = \{e\}$. Since the period group $N$ was arbitrary, this shows that $f$ is aperiodic.
\end{proof}
Now we claim:
\begin{proposition}\label{MonotoneConv} Let $g \in H$. Then $T(g) > 0$ if and only if $g.x > x$ for all $x \in \R$.
\end{proposition}
\begin{proof} Suppose $g \in H$ satisfies $T(g) > 0$. If $g$ had a fixed point $x$, then $g^n.x = x$ for all $n \in \mathbb N$ and hence $T(g) = 0$ by definition. Thus $g$ cannot have a fixed point. If $g.x < x$ for some $x$, then by monotonicity $g^n.x \leq g^{n-1}.x$, hence $g^{n}.x < x$ by induction and thus we get the contradiction
\[T(g) = \lim_{n \to \infty}\frac{g^n.x-x}{n} \leq 0.\]
Thus $g^n.x > x$ for all $x \in \R$. Conversely assume $g.x > x$ for all $x \in \R$. Reversing inequality signs in the above argument we get $T(g) \geq 0$. It thus remains only to show that $T(g) \neq 0$. Assume $T(g) = 0$ for contradiction; it will suffice to show that
\begin{eqnarray}\label{InductiveCircle}
g^n.x \leq x+1
\end{eqnarray} 
for all $n \in \mathbb N$. Indeed, \eqref{InductiveCircle} implies that the monotone sequence $g^n.x$ is bounded, and therefore converges to a fixed point of $g$, which yields the desired contradiction. We now prove \eqref{InductiveCircle} assuming $T(g) = 0$: Suppose $g^{n_0}(x) > x+1$ for some $n_0 \in \mathbb N$; since $g$ is monotone and commutes with integral translations we have
\[g^{2n_0}(x) = g^{n_0}(g^{n_0}(x)) \geq g^{n_0}(x+1) = g^{n_0}(x)+1 > x+2,\]
and inductively we obtain $g^{mn_0}(x) > x+m$ for every $m \in \mathbb N$. This in turn implies $T(g) \geq \frac{1}{n_0} > 0$ contradicting $T(g) = 0$.
\end{proof}
\begin{corollary}\label{DominantsH} Let $G$ be a subgroup of $H$ for which $T|_G \not\equiv 0$. Then the unique maximal element $G^{++}_{\max}$ of $\mathcal D_{T|_G}(G)$ is given by
\begin{eqnarray}
G^{++}_{\max} := \{g \in G\,|\, T(g) > 0\} = \{g \in G\,|\, \forall x \in \R:\, g.x > x\}.
\end{eqnarray}
\end{corollary}
\begin{proof} The equality of the last two sets follows from Proposition \ref{MonotoneConv}. Let us denote this set by $S$; from the second description it follows immediately, that $S$ is a semigroup. On the other hand, the first description yields properties (D1)--(D3) from Corollary \ref{DominantsChara}. Thus $S \in \mathcal D_T(G)$, and maximality follows from the first description and Proposition \ref{DominantsFromOSG}.
\end{proof}
For any $y \in \R$ we denote by $\tau_y: \R \to \R$ the translation map $\tau_y(x) := x+y$. Since these commute with integer translations we have $\tau_y \in H$ for all $y \in \R$. A subgroup $G<H$ is said to \emph{contain small translations}, if for every $\delta > 0$ there exists $\epsilon > 0$ with $\epsilon < \delta$ and $\tau_\epsilon \in G$. Then we have:
\begin{proposition}\label{SubgroupsOfH} Let $G$ be a subgroup of $H$ containing small translations such that the restriction $T|_G$ is aperiodic. Then the unique maximal element of $\mathcal M_{T|_G}(G)$ is given by
\[G^{+}_{\max} := \{g \in G \,|\, \forall x \in \R:\, g.x \geq x\}.\]
\end{proposition}
\begin{proof}
By Corollary \ref{DominantsH} and Corollary \ref{DomIncl} the dominant set of the maximal element of  $\mathcal M_{T_G}(G)$ is  given by.
\begin{eqnarray*}
G^{++}_{\max} := \{g \in G\,|\, T(g) > 0\} = \{g \in G \,|\, \forall x \in \R:\, g.x > x\}.
\end{eqnarray*}
By Proposition \ref{MaximalFromDom} it thus remains only to show that
\begin{eqnarray}\label{MaximalSLTrick}G^+_{\max} = \{g \in G\,|\,
gG^{++}_{\max} \subset G^{++}_{\max} \}.\end{eqnarray}
The inclusion $\subseteq$ is obvious. For the other inclusion we argue by contradicition: Assume $gG^{++}_{\max} \subset G^{++}_{\max}$, but $g.x -x \leq -\delta < 0$ for some $x \in \R$ and $\delta > 0$. Choose  $\epsilon > 0$ with $\epsilon < \delta$ and $\tau_\epsilon \in G$ and put $h := g^{-1}\tau_{\epsilon}g$; then $h \in G^{++}_{\max}$, since
$\tau_{\epsilon}$ has translation number $\epsilon > 0$ and $G^{++}_{\max}$ is conjugation-invariant. By assumption, this implies $gh \in G^{++}_{\max}$. On the other hand
\[(gh).x = \tau_{\epsilon}(g.x) = g.x + \epsilon = x + (gx - x) + \epsilon \leq x,\]
which is a contradiction. This establishes \eqref{MaximalSLTrick} and finishes the proof.
\end{proof}
For $G = H$ we obtain:
\begin{corollary}
The maximal order on the universal covering $H$ of ${\rm Homeo}^+(S^1)$ with respect to the translation number is given by
\[g \geq h \Leftrightarrow \forall x \in \R: g.x \geq h.x \quad(g,h \in H).\]
\end{corollary}
For another important special case of Proposition \ref{SubgroupsOfH} see Section \ref{SecSL2case} below.

\subsection{An abstract criterion for maximality}
In the above examples, the maximal dominant sets have been open. Conversely, open dominant sets in connected groups tend to be maximal. The following theorem makes this statement precise. For the statement we refer to a continuous map $T: \R^{\geq 0} \to G$ as one-parameter semigroup if $T(t+s) = T(t)T(s)$ and $T(0) = e$.
\begin{theorem}\label{Recognition} Let $G$ be a connected topological group and $f \in \mathcal{HQM}({G}) \setminus \{0\}$ be aperiodic. Suppose $G^+ \in \mathcal M_f(G)$ is closed and satisfies:
\begin{itemize}
\item[$(\dagger)$] The dominant set $G^{++}$ of $G^+$ is open in $G$.
\item[$(\dagger\dagger)$] ${\rm Int}(G^+)$ is path-connected, dense in $G^+$ and there exists a one-parameter semigroup $T: \R^{\geq 0} \to G$ with $T(\R^{>0})\subset {\rm Int}(G^+)$.
\end{itemize}
Then $G^+ = G^+_{\max}$ is the unique maximal element of $\mathcal M_f(G)$.
\end{theorem}
Note that in view of Lemma \ref{TrivialTopologicalStuff} assumption $(\dagger)$ implies
 \begin{eqnarray}\label{OpenDom1}
 G^{++} = {\rm Int}(G^+) \subset {\rm Int}(G^{+}_{\max}) \subset G^{++}_{\max},
 \end{eqnarray}
 where as before $G^{++}_{\max}$ denotes the dominant set of $G^+_{\max}$, or equivalently, the maximal object in $\mathcal D_f(G)$. On the other hand Lemma \ref{DominantsFromOSG} yields
 \begin{eqnarray}\label{OpenDom2}
 G^{++} = G_{\max}^{++} \cap G^+.
 \end{eqnarray}
Moreover we observe:
\begin{lemma}\label{RayExistence} Let $G$, $f$ as in Theorem \ref{Recognition} and suppose $G^+ \in \mathcal M_f(G)$ satisfies $(\dagger\dagger)$. Let $T: \R^{\geq 0} \to G$ be a one-parameter semigroup with $T(\R^{>0})\subset {\rm Int}(G^+)$. Then $f(T(t)) \to \infty$ as $t \to \infty$.
\end{lemma}
\begin{proof} Fix $t_0 > 0$ and observe that by Lemma \ref{TrivialTopologicalStuff} we have $T(t_0)\in G^{++}$. In particular,
 $f(T({t_0})) > 0$ by Lemma \ref{DominantsFromOSG}. Now $H := \{T(t)\,|\,t \geq 0\} \cup \{T(t)^{-1}\,|\,t \geq 0\}$ is an abelian subgroup of $G$, hence $f$ restricts to a homomorphism on $H$. As $f(T({t_0})) > 0$ we see that $f(T(t)) \to \infty$ as $t \to \infty$.
\end{proof}
From this we deduce:
\begin{corollary}\label{RecognitionMain} Let $G$, $f$ as in Theorem \ref{Recognition}. If $G^+ \in \mathcal M_f(G)$ is closed and satisfies $(\dagger\dagger)$, then $G^{+}_{\max}$ satisfies $(\dagger\dagger)$.
\end{corollary}
\begin{proof} Since $e \in G^+$ and ${\rm Int}(G^+)$ is dense in $G^+$, we see that $e$ is an accumulation point of ${\rm Int}(G^+)$. We may thus choose a net $x_i \in {\rm Int}(G^+)$ with $x_i \to e$; then in particular $x_i \in {\rm Int}(G^+_{\max})$. Now the latter is an ideal in $G^{+}_{\max}$; thus for every $g \in G^{+}_{\max}$ we have $gx_i \in  {\rm Int}(G^{+}_{\max})$ and $gx_i \to g$. This shows that ${\rm Int}(G^{+}_{\max})$ is dense in $G^{+}_{\max}$. Now choose $T:  \R^{\geq 0} \to G$ as in Lemma \ref{RayExistence}. Then for every $g \in {\rm Int}(G^{+}_{\max})$ the curve $\gamma_g(t) := gT(t)$ is contained in ${\rm Int}(G^{+}_{\max})$ and satisfies $f(\gamma_g(t)) \to \infty$. By Lemma \ref{TrivialTopologicalStuff} we thus find $t_0 \in \R$ with $\gamma_g(t_0) \in {\rm Int}(G^+)$. This shows that every $g \in {\rm Int}(G^{+}_{\max})$ can be connected by a curve inside ${\rm Int}(G^{+}_{\max})$ to an element in ${\rm Int}(G^+)$; since the latter is path-connected, we deduce that also ${\rm Int}(G^{+}_{\max})$ is path-connected. Finally, every one-parameter semigroup contained in ${\rm Int}(G^+)$ is contained in ${\rm Int}(G^+_{\max})$.
\end{proof}
%\begin{corollary}\label{RecognitionMain}
%Under the assumptions of Theorem \ref{Recognition} we have
%\begin{itemize}
%\item[(i)] $G^{++} = {\rm Int}(G^+)$ is connected and dense in $G^+$.
%\item[(ii)] ${\rm Int}(G^{+}_{\max})$ is connected and dense in $G^{+}_{\max}$.
%\end{itemize}
%\end{corollary}
%\begin{proof} (i) follows from \eqref{OpenDom1} and assumption $(\dagger\dagger)$. We see in particular, that $e$ is an accumulation point of $G^{++}$. We may thus choose a net $x_i \in G^{++}$ with $x_i \to e$. Now $G^{++} \subset {\rm Int}(G^{+}_{\max})$ by \eqref{OpenDom1} and the latter is an ideal in $G^{+}_{\max}$; in particular, for every $g \in G^{+}_{\max}$ we have $gx_i \in  {\rm Int}(G^{+}_{\max})$ and $gx_i \to g$. This shows that ${\rm Int}(G^{+}_{\max})$ is dense in $G^{+}_{\max}$. Now choose $T:  \R^{\geq 0} \to G$ as in Lemma \ref{RayExistence}. Then for every $g \in {\rm Int}(G^{+}_{\max})$ the curve $\gamma_g(t) := gT(t)$ is contained in ${\rm Int}(G^{+}_{\max})$ and satisfies $f(\gamma_g(t)) \to \infty$. We thus find $t_0 \in \R$ with $\gamma_g(t_0) \in G^{++}$. This shows that every $g \in {\rm Int}(G^{+}_{\max})$ can be connected by a curve inside ${\rm Int}(G^{+}_{\max})$ to an element in $G^{++}$. Since $G^{++}$ is a connected subset of ${\rm Int}(G^{+}_{\max})$, this implies (ii).
%\end{proof}
We will combine this observation with the following elementary lemma from point-set topology:
\begin{lemma}\label{pointset1}
Let $X$ be a topological space with non-empty subsets $A,B$. Assume that ${\rm Int}_X(B)$ is connected and that
\[A = \overline{{\rm Int}_X(A)} \subsetneq B = \overline{{\rm Int}_X(B)}.\]
Then
\[\partial A \cap {\rm Int}_X(B) \neq \emptyset.\]
\end{lemma}
\begin{proof} Abbreviate $A^\circ := {\rm Int}_X(A)$, $B^\circ = {\rm Int}_X(B)$ and observe that $A^\circ \subset B^\circ$ is an open subset. Now assume $\partial A \cap B^\circ = \emptyset$ for contradiction and  denote by $cl_{B^\circ} A^\circ$ the closure of $A^\circ$ in $B^\circ$. Then
\[cl_{B^\circ} A^\circ = \overline{A^\circ} \cap B^\circ = A \cap B^\circ = (A^\circ \cup \partial A) \cap B^\circ = (A^\circ \cap B^\circ)  \cup (\partial A \cap B^\circ) = A^\circ.\]
This shows that $A^\circ$ is both closed and open in $B^\circ$. Since $B^\circ$ is assumed connected we either have  $A^\circ = \emptyset$ or $A^\circ = B^\circ$. Taking closure we thus end up with either of the two contradictions $A =
\emptyset$ or $A = B$.
\end{proof}
\begin{proof}[Proof of Theorem \ref{Recognition}]
In view of Corollary \ref{RecognitionMain} we can apply Lemma \ref{pointset1} with $A := G^+$, $B := G_{\max}^+$ and $X := G$. If we assume $A \neq B$, then the lemma implies
\[\partial G^+ \cap {\rm Int}(G_{\max}^+) \neq \emptyset.\]
Now observe that by \eqref{OpenDom1} we have $G^{++} = {\rm Int}(G^+)$. In particular, $\partial G^+ = G^+ \setminus G^{++}$. On the other hand Lemma \ref{TrivialTopologicalStuff} yields $ {\rm Int}(G_{\max}^+) \subset G_{\max}^{++}$. We thus obtain
\[(G^+ \setminus G^{++}) \cap G_{\max}^{++} \neq \emptyset,\]
which contradicts \eqref{OpenDom2}.
\end{proof}

\section{Maximal and continuous orders on Lie groups}\label{SecLie}

\subsection{Classification of aperiodic quasimorphisms on finite-dimensional Lie groups}
In this subsection we provide a classification of quasimorphisms on finite-dimensional connected Lie groups. (Recall that for us a quasimorphism is by definition assumed to be continuous.)  By Proposition \ref{Factorization1}  we may restrict attention to aperiodic quasimorphisms, whose
classification is an easy consequence of results from \cite{BuMo, Shtern, Surface} and probably known to people working on bounded cohomology. However, to the best of our knowledge the explicit classification statement has never appeared in print and certainly is not widely known among Lie theorists. We therefore explain the classification in some details, starting from the following result:
\begin{proposition}[Burger--Monod, Shtern, \cite{BuMo, Shtern}]\label{BMS} Let $G$ be a connected semisimple Lie group with Lie algebra $\L g$, $\L k$ a maximal compact Lie subalgebra of $\L g$, $K$ the corresponding analytic subgroup of $G$ and $Z(G)$ and $Z(K)$ the centers of $G$ and $K$ respectively.
\begin{itemize}
\item[(i)] $Z(G)$ is finite if and only if $Z(K)$ is compact. In this case, $G$ does not admit a nonzero homogeneous quasimorphism.
\item[(ii)] If $G$ is simple and $Z(G)$ has infinite center, then the space of homogeneous quasimorphisms on $G$ is one-dimensional. This statement remains true even if $G$ is considered as a discrete group.
%\item[(iii)] If $G= G_1 \times \dots \times G_m$ is simply-connected with simple factors $G_j$ then every homogeneous quasimorphism on $G$ is of the form
%\[f(g_1, \dots, g_m) = f_1(g_1) + \dots + f_m(g_m)\]
%with $f_j$ a homogeneous quasimorphism on $G_j$.
\end{itemize}
\end{proposition}
Let us now describe the quasimorphisms appearing in (ii) explicitly: A connected simple real Lie group $G$ can only have infinite center if the associated symmetric space admits an invariant complex structure; we then call $G$ a \emph{Hermitian} Lie group. Thus assume that $G_0$ is an adjoint simple Hermitian Lie group and fix an Iwasawa decomposition $G_0 = K_0AN$; then the universal covering $\widetilde{G}$ of $G_0$ has a compatible decomposition of the form $\widetilde{G}  = \widetilde{K}AN$, where $\widetilde{K}$ now has a one-dimensional non-compact center $Z$. Fix an isomorphism $Z\cong \R$ and denote by $\pi$ the projection map
\[\pi: \widetilde{G} = \widetilde{K}AN \to Z \cong \R.\]
Then the homogeneization
\begin{eqnarray}\label{GW}
\mu_{\widetilde{G}}: \widetilde{G} \to \R, \quad g \mapsto \lim_{n \to \infty} \frac{\pi(g^n)}{n}\end{eqnarray}
of $\pi$ defines a homogeneous quasimorphism on $\widetilde{G}$, called the \emph{Guichardet-Wigner quasimorphism} of $\widetilde{G}$ \cite{GuichardetWigner, Shtern}. We warn the reader that due to the homogeneization process involved in its definition, $\mu_{\widetilde{G}}$ does not respect the Iwasawa decomposition in any reasonable way. In fact, the above definition is rather useless for practical computations. To actually compute $\mu_{\widetilde{G}}$ one has to use the refined Jordan decomposition of $G_0$; see \cite{Surface} for details.\\

By Proposition \ref{BMS} every homogeneous quasimorphism on $\widetilde{G}$ is a multiple of $\mu_{\widetilde{G}}$; moreover $\mu_{\widetilde{G}}$ descends to a homogeneous quasimorphism $\mu_G$ on every finite central quotient $G$ of $\widetilde{G}$, but not to any infinite quotient. In particular, $\mu_{\widetilde{G}}$ descends to an aperiodic homogeneous quasimorphism $\mu_{{G}}$ on $G := \widetilde{G}/\pi_1(G_0)_{\rm tors}$. By a slight abuse of language we will refer to $\mu_G$ as an \emph{aperiodic Guichardet-Wigner quasimorphism}. This terminology understood, the aperiodic quasimorphism on simple Lie groups are precisely the aperiodic Guichardet-Wigner quasimorphisms. This classification result can be extended to the semisimple case using the following simple observation:
\begin{lemma}
Let $G = G_1 \times G_2$ be a direct product of topological groups and $f: G \to \R$ a homogeneous quasimorphism. Then there exist homogeneous quasimorphisms $f_j: G_j \to \R$ such that $f(g_1, g_2) = f_1(g_1)+f_2(g_2)$.
\end{lemma}
\begin{proof} Set $f_j := f|_{G_j}$ and let $g_1 \in G_1$, $g_2 \in G_2$. The subgroup of $G$ generated by $g_1$ and $g_2$ is abelian, hence $f$ restricts to a homomorphism on this subgroup. In particular, $f(g_1, g_2) = f(g_1)+f(g_2) = f_1(g_1)+f_2(g_2)$.
\end{proof}
Thus if $G= G_1 \times \dots \times G_m$ is simply-connected semisimple with simple factors $G_j$ and $G_1, \dots, G_l$ are Hermitian, while $G_{l+1}, \dots, G_m$ are not, then the space of homogeneous quasimorphism on $G$ is spanned by the pullbacks of the Guichardet-Wigner quasimorphisms of $G_1, \dots G_l$ to $G$. In particular, every homogeneous quasimorphism on $G$ factors through a homogeneous quasimorphism on $G_1 \times \dots \times G_l$, and $f$ can be aperiodic only if all almost simple factors of $G$ are Hermitian. From this observation a classification of aperiodic homogeneous quasimorphisms on semisimple groups is immediate. Indeed, assume $f|_{G}$ is aperiodic; then then universal covering $\widetilde{G}$ of $G$ is of the form
$\widetilde{G}= G_1 \times \dots \times G_l$ with $G_j$ simply-connected Hermitian simple. Moreover, $G = \widetilde{G}/\Gamma$, where
\[\Gamma :=\{g \in Z(G_1)\times \dots \times Z(G_l)\,|\,f(g) = 0\}.\]
The general classification is reduced to the semisimple case by means of the following observation:
\begin{proposition}\label{ClassifMain}
Let $G$ be a connected finite-dimensional Lie group and $f: G \to \R$ be an aperiodic quasimorphism. Then $G$ is reductive and the center of $G$ is at most one-dimensional.
\end{proposition}
\begin{proof} We use the basic fact that the restriction of a homogeneous quasimorphism to an amenable group is a homomorphism. We claim that this implies that the radical $RG$ of $G$ has dimension $\leq 1$. Assume otherwise; then $f|_{RG}$ is a homomorphism since $RG$ is amenable. Since $\dim RG \geq 2$ there is a non-trivial connected normal subgroup $H$ of codimension $1$ in $RG$, on which $f$ vanishes. If $H$ is normal in $G$, then it is a period subgroup of $f$. Otherwise there exists $g \in G$ such that $gHg^{-1} \neq H$. Now denote by $\L h$ and $\L r$ the Lie algebras of $H$ and $RG$ respectively; for dimension reasons we have $\L r = \L h + {\rm Ad}(g)(\L h)$. Since $f_0 = f|_{RG}$ is a homomorphism into $\R$, it is smooth with ${\rm Ad}$-invariant derivative $df_0$. Since $df_0|_{\L h} = 0$ this implies $df_0|_{\L r} = 0$. This in turn means that $RG$ is a period group for $f$ in this case. In any case, $f$ cannot be aperiodic. This contradiction establishes $\dim RG \leq 1$. If $RG$ is trivial, then $G$ is semisimple.
Otherwise the universal cover of $G$ is a semidirect product of $\R$ and a semisimple group. Since a semisimple group does not admit a one-dimensional non-trivial representation, this semidirect product is in fact direct. This shows that $G$ is reductive also in this case.
\end{proof}
Combining the previous observations and Proposition \ref{Factorization1} we finally obtain the following result:
\begin{theorem}\label{Classification}
Let $H$ be a connected finite-dimensional Lie group and $f:H \to \R$ a homogeneous quasimorphism. Then $f$ factors uniquely as
\[H \xto{p} G \xto{f_0} \R,\]
where $p$ is a continuous homomorphism of Lie groups and $f_0$ is an aperiodic homogeneous quasimorphism. The universal covering $\widetilde{G}$ of $G$ is of the form
\[\widetilde{G} = H \times G_1 \times \dots\times G_m,\]
where $H$ is either trivial or isomorphism to $\R$ and $G_1,\dots, G_m$ are simple Hermitian Lie groups. Moreover, the lift of $f_0$ to $\widetilde{G}$ is given by
\[(h, g_1 \dots, g_n) \mapsto f_H(h) + f_1(g_1) + \dots + f_m(g_m),\]
where $f_H$ is either trivial or an isomorphism and $f_j$ is some multiple of the Guichardet-Wigner quasimorphism on $G_j$ for $j= 1, \dots, m$.
\end{theorem}
Theorem \ref{Classification} reduces Problem \ref{MainProblem} for connected finite-dimensional Lie groups to the study of maximal orders associated with linear combinations of Guichardet-Wigner quasimorphisms. Here we will focus on the following subproblem:
\begin{problem*}\label{LieProblem} Describe explicitly the maximal orders associated to aperiodic Guichardet-Wigner quasimorphism.
\end{problem*}
In the remainder of this section we will obtain a complete answer to this problem for the universal covering of $SL_2(\R)$ and a partial answer in the general case.

\subsection{The case of simple Lie groups}
We now turn to the study of Problem \ref{LieProblem}. Thus let $G_0$ be a connected simple adjoint Hermitian Lie group, $\widetilde{G}$ its universal cover and $G := \widetilde{G}/\pi_1(G_0)_{\rm tors}$. We observe that $G \to G_0$ is an infinite cyclic covering, while $\widetilde{G} \to G$ is a finite covering. All aperiodic homogeneous quasimorphisms on $G$ are of the form $f = \lambda \cdot \mu_G$, where $\mu_G:G \to R$ is the aperiodic Guichardet-Wigner quasimorphism as given by \eqref{GW}. We want to determine the maximal orders corresponding to these quasimorphisms. Since the maximal order corresponding to a quasimorphism is invariant under taking positive multiples, it suffices to consider $\pm \mu_G$; moreover, if $G^+_{\max}$ denotes the maximal order semigroup for $\mu_G$, then the maximal order semigroup for $-\mu_G$ is given by $(G^+_{\max})^{-1}$. It thus suffices to determine $G^+_{\max}$. This semigroup has the following strong maximality property:
\begin{proposition}
The maximal order semigroup $G^+_{\max}$ associated with $\mu_G$ is a maximal conjugation-invariant semigroup of $G$ in the sense that there does not exist a conjugation-invariant semigroup $S$ with $G^+_{\max} \subsetneq S \subsetneq G$.
\end{proposition}
\begin{proof} If $G^+_{\max} \subsetneq S$, then $S$ cannot be pointed by Proposition \ref{PointedMaximality}. Thus $H := S \cap S^{-1}$ is a non-trivial normal subgroup of $G$. Since $G$ is simple, it either coincides with $G$ or is discrete. Assume the latter; then $H$ is central in $G$. However, $Z(G) = \ker(G \to G_0) \cong \Z$. Since $H$ is a non-trivial subgroup, it must thus be of finite index in $Z(G)$. We deduce that $S$ contains an element $g$ with $\mu_{G}(g) < 0$. Now every $x \in G$ can be written as $x = g^{-n}g^nx$ and for any $C> 0$ we can chosse $n$ so large that $\mu_{G}(g^nx) > C$. We can thus ensure that $g^nx \in \hat G^+_{\max} \subset S$ by choosing $n$ large enough; this, however, implies $x = g^{-n}g^nx \in S$, so $S = G$ and hence $H=G$ contradicting discreteness. We thus have $S = G$ whenever $G^+_{\max} \subsetneq S$.
\end{proof}
We now determine the shape of $G^+_{\max}$ at least infinitesimally. For this we compare $G^+_{\max}$ with continuous orders on $G$. Recall that a bi-invariant order on a Lie group $G$ is called \emph{continuous} if its order semigroup is closed and locally topologically generated. By a result of Neeb \cite{Neeb} this implies that $G^+$ is a \emph{Lie semigroup}. This means that $G^+$ can be reconstructed from its Lie wedge
\[{\bf L}(G^+) := \{X \in \L g\,|\, \forall t >0: \;\exp(tX) \in G^+\}\]
as the closure of the semigroup generated by $\exp({\bf L}(G^+))$, i.e.
\[G^+ = \overline{\langle \exp({\bf L}(G^+))\rangle}.\]
Now let $G$ be a Hermitian Lie group admitting an aperiodic quasimorphism. By results of Ol{\cprime}shanski{\u\i} \cite{Olshanski} there is a unique (up to inversion) maximal continuous bi-invariant order on $G$, which we denote by $G^+_{cont}$. The Lie wedge $C^+ := {\bf L}(G^+_{cont})$ is an ${\rm Ad}$-invariant cone in $\L g$. Ol{\cprime}shanski{\u\i} has proved\footnote{Strictly speaking, Ol{\cprime}shanski{\u\i} always works with simply-connected groups, but it is easy to see that his results carry over to the case considered here. See \cite{HartnickDiss} for details.} that this cone is the  maximal ${\rm Ad}$-invariant pointed cone in $\L g$ if and only if $G$ is of tube type; in the non-tube type case he proved that for every ${\rm Ad}$-invariant cone $C \supsetneq C^+$ we have
\begin{eqnarray}\label{OlshanskiMaximality}
 \overline{\langle \exp(C)\rangle} = G.
\end{eqnarray}
The latter results thus holds independent of whether $G$ is of tube type or not. Now we can prove:
\begin{proposition}\label{SameWedge} The semigroups $G^+_{cont}$ and $G^+_{\max}$ have the same Lie wedge $C^+$. Moreover,  $G^+_{cont} \subseteq G^+_{\max}$.
\end{proposition}
\begin{proof} It was established in \cite[Lemma 3.4]{BSH} that $ G^+_{cont} \in \mathcal M_{\mu_{  G}}(  G)$. Consequently, $  G^+_{cont} \subseteq   G^+_{\max}$ and thus ${\bf L}(  G^+_{cont}) \subseteq {\bf L}(  G^+_{\max})$. If the inclusion was proper, then $C := {\bf L}(  G^+_{\max})$ would satisfy \eqref{OlshanskiMaximality}, contradicting the pointedness of $  G^+_{\max}$.
\end{proof}
As mentioned in the introduction, we believe that $G^+_{cont}$ and $G^+_{\max}$ coincide not only infinitesimally, but even globally. For $SL_2(\R)$ we present a proof in the next section. The following proposition collects some evidence for our conjecture in the general case:
\begin{proposition}\label{Evidence}
The semigroups $G^+_{cont}$ and $G^+_{\max}$ share the following properties:
\begin{itemize}
\item[(i)] $G^+_{cont}$ and $G^+_{\max}$ have the same Lie wedge.
\item[(ii)] $G^+_{cont}$ and $G^+_{\max}$ are path-connected.
\item[(iii)] $G^+_{cont}$ and $G^+_{\max}$ have dense path-connected interiors.
\item[(iv)] $G^+_{cont}$ and $G^+_{\max}$ are closed.
\item[(v)] Let $g \in G$; then $g^n \in {\rm Int}(G^+_{cont})$ for some $n \in \mathbb N$, if and only if $g^m \in {\rm Int}(G^+_{\max})$ for some $m \in \mathbb N$.
\end{itemize}
\end{proposition}
The proof uses the following fundamental results from Lie semigroup theory \cite[Cor. 3.11 and Prop. 3.13]{HiNe}:
\begin{lemma}\label{Smoothness}
Let $S$ be a conjugation-invariant Lie semigroup in a simple Lie group. Then ${\rm Int}(S)$ is path-connected and dense in $S$.
\end{lemma}
\begin{proof}[Proof of Proposition \ref{Evidence}] (i) was established in Proposition \ref{SameWedge}. (iii) for $G^+_{cont}$ is a special case of Lemma \ref{Smoothness}; for $G^+_{\max}$ the corresponding properties then follow from Corollary \ref{RecognitionMain}. (ii) is a consequence of (iii), since the closure of a path-connected open subset of a manifold is path-connected. (iv) is obvious from the definitions. Finally, by Lemma \ref{TrivialTopologicalStuff} and Lemma \ref{DominantsFromOSG} both conditions in (v) are equivalent to $\mu_G(g) > 0$.
\end{proof}
We do not know any example of a closed semigroup of $G$ besides $G^+_{cont}$, which has dense path-connected interior and shares both the Lie wedge and the divisible hull of the interior with $G^+_{cont}$. It thus seems possible that Proposition \ref{Evidence} already implies Conjecture \ref{MainConjecture}. In any case, we do not see how to  prove this.

\subsection{Orders on the universal covering of $SL_2(\R)$}\label{SecSL2case}

Throughout this subsection let $G$ denote the universal covering group of $SL_2(\R)$. Our goal is still to describe the maximal order semigroup $G^+_{\max}$  associated with the Guichardet-Wigner quasimorphism on $G$; we first provide a geometric description. To this end we observe that the group $SL_2(\R)$ acts on the circle, extending its isometric action on the Poincar\'e disc. The corresponding homomorphism into ${\rm Homeo}^+(S^1)$ lifts to an embedding $G \hookrightarrow H$, hence the translation quasimorphism restricts to a homogeneous quasimorphism on $G$, which is nonzero, since it does not vanish on the universal cover of the rotation group. By the classification, this restriction is a multiple of the Guichardet-Wigner quasimorphism, and we can choose our sign in such a way that it is a positive multiple. Then $G^+_{\max}$ is the maximal order semigroup associated with the restriction of the translation number, and we obtain the following special case of Proposition \ref{SubgroupsOfH}:
\begin{corollary}\label{SL2Max1} The unique maximal element of $\mathcal M_{T_G}(G)$ is given by
\[G^{+}_{\max} := \{g \in G \,|\, \forall x \in \R:\, g.x \geq x\};\]
its dominant set is given by
\begin{eqnarray*}
G^{++}_{\max} := \{g \in G\,|\, T(g) > 0\} = \{g \in G \,|\, \forall x \in \R:\, h.x > x\}.
\end{eqnarray*}
\end{corollary}
In the terminology of the introduction this states that the maximal order on the universal covering of $SL_2(\R)$ coincides with the geometric one. We now aim to show that both coincide with $G^+_{cont}$. In view of Lemma \ref{Smoothness} the maximality criterion from Theorem \ref{Recognition}  reads as follows:
\begin{corollary}\label{RecognitionLie} If $G^+ \in \mathcal M_{\mu_G}(G)$ is a Lie semigroup whose dominant set $G^{++}$ is open in $G$, then $G^+ = G^+_{\max}$.
\end{corollary}
In order to apply this in the case at hand we need the following crucial observation:
\begin{lemma}[Hilgert-Hofmann]\label{HilgertHofmann}
The universal covering $G$ of $SL_2(\R)$ satisfies
\begin{eqnarray}
G = {\rm Int}(G^{+}_{cont}) \cup {\rm Int}((G^{+}_{cont})^{-1}) \cup \exp_G(\L g).
\end{eqnarray}
\end{lemma}
\begin{proof} This can be seen directly from \cite[Figure 3]{HilgertHofmann}.
\end{proof}
From this we deduce:
\begin{proposition}\label{SL2case} The maximal order $G^+_{\max} $ on $G$ with respect to the Guichardet-Wigner quasimorphism $\mu_G$ coincides with the maximal continuous order $G^+_{cont}$ on $G$, on which $\mu_G$ is non-negative.
\end{proposition}
It would probably be possible to establish the proposition by writing out all the objects involved in explicit formulas. However, we prefer to give a conceptual proof, which demonstrates some of the machinery developed in this article:
\begin{proof}[Proof of Proposition \ref{SL2case}]
We first claim that the dominant set $G^{++}_{cont}$ of $G^{+}_{cont}$  coincides with the dominant set $G^{++}_{\max}$ of $G^+_{\max}$. We recall from Corollary \ref{SL2Max1} that the latter is given by
\[G^{++}_{\max} = \{g \in G\,|\, \mu_G(g) > 0\},\]
hence the inclusion $G^{++}_{cont} \subseteq G^{++}_{\max}$ is an immediate consequence of Lemma \ref{DominantsFromOSG}. For the other inclusion we use Lemma \ref{HilgertHofmann}: If $g \in {\rm Int}(G^+)$, then $g \in G^{++}_{cont}$ by Lemma \ref{TrivialTopologicalStuff}. If $g \in {\rm Int}((G^+)^{-1})$, then similarly $g^{-1} \in G^{++}_{cont}$, whence $\mu(g) = -\mu(g^{-1})< 0$ by Proposition \ref{DominantsFromOSG}. It thus remains only to show that for $g \in \exp_G(\L g)$ we have $g \in G^{++}_{cont}$ provided $\mu_G(g) > 0$. For this let $g = \exp_G(X)$, $X \in \L g$. Then $\mu_G(g) >0$ implies that $\mu_G$ is non-trivial on the one-parameter group $\gamma_X(t) := (\exp(tX))$. Since the group $\{\gamma_X(t)\}$ is amenable, the restriction of $\mu_G$ to this group is a homomorphism. In particular, $\mu_G(\gamma_X(t)) > 0$ for all $t > 0$. We deduce that $\gamma_X(t) \in G^+_{\max}$ for all $t > 0$, whence $X \in {\bf L}(G^{+}_{\max})$. Then Proposition \ref{SameWedge} yields $X \in {\bf L}(G^{+}_{cont})$, whence $g \in G^+_{cont}$. Taking into account Proposition \ref{DominantsFromOSG} we deduce the claim. Since $\mu_G$ is continuous, it now follows that $G^{++}_{cont}$ is open. Thus Corollary \ref{RecognitionLie} applies and yields the proposition.
\end{proof}
As immediate consequences (of the proof) we obtain the first part of the following corollary; its second part can then be deduced by taking another look at \cite[Figure 3]{HilgertHofmann}.
\begin{corollary}
For the universal covering $G$ of $SL_2(\R)$ we have
\begin{eqnarray}\label{SL2Formulas}
G^+_{cont} = \{g \in G\,|\, \forall h \in G: \mu_G(gh) \geq \mu_G(h)\} = \overline{\{g \in G\,|\, \mu_G(g) > 0\}}.
\end{eqnarray}
Moreover, the zero set of $\mu_G$ in $G$ has dense interior, and its complement has two connected components, given by the interiors of $G^+_{cont}$ and $(G^+_{cont})^{-1}$ respectively.
\end{corollary}
As another consequence we obtain:
\begin{corollary}
$G^+_{cont}$ and $(G^+_{cont})^{-1}$ are maximal conjugation-invariant subsemigroups of $G$.
\end{corollary}
It is easy to see that for $G$ not locally isomorphic to $SL_2(\R)$ there will always exist elements with $\mu_G(g) > 0$, which are not contained in $G^{+}_{\max}$. Indeed, this can be seen already by considering a compact Cartan subgroup. Thus the second description of $G^+_{cont}$ in \eqref{SL2Formulas} is really special to the universal covering of $SL_2(\R)$. On the other hand, the first description just arises from the equality $G^{+}_{\max} = G^+_{cont}$, so it has a chance to be generalized to more general groups.

\section{Comparison to the geometric order}\label{OtherOrders}

\subsection{The geometric order and the maximal order} We return to the general case, where $G_0$ is an arbitrary adjoint simple Hermitian Lie group and $G$ denotes the quotient of the universal covering of $G_0$ by the torsion subgroup of $\pi_1(G_0)$. Then $G_0$ can be realized as the biholomorphism group of an irreducible bounded symmetric domain $\mathcal D$ and thus acts on the Shilov boundary $\check S$ of $\mathcal D$. This action induces a transitive, effective action of $G$ on the universal covering $\check R$ of $\check S$. Now assume that the bounded symmetric domain $\mathcal D$ is of tube type. (By abuse of language we also say that $G$ is of tube type in this case.) Then $\check R$ admits a $G$-invariant partial order. In fact, two slightly different orders are described in \cite{ClercKoufany, BSH2}; the order used in \cite{BSH2} is the closure of the order described in \cite{ClercKoufany}. We decide to work with the closed order here. To avoid confusion, let us spell out the definition explicitly: There is a unique up to inversion $G$-invariant field of closed cones $\mathcal C_x \subset T_x\check R$ on $\check R$, and a piecewise $C^\infty$-curve $\gamma: [0,1] \to \check R$ will be called causal if $\dot \gamma(t) \in \mathcal C_{\gamma(t)}$ whenever it is defined. In this case we write $\gamma(0) \preceq_s \gamma(1)$; now the order on $\check R$ we refer to is the closure $\preceq$ of $\preceq_s$ in $\check R \times \check R$. We call it the \emph{Kaneyuki order}, since the causal structure on $\check R$ is the lift of the causal structure on $\check S$ constructed by Kaneyuki \cite{Kaneyuki}. This order induces a partial order $\leq$ on $G$ via
\[g \leq h :\Leftrightarrow \forall x \in \check R: g.x \preceq g.y,\]
whose order semigroup we denote by $G^+_{geom}$. In the case of the symplectic group this construction is classical Here $\check S$ is the Lagrangian Grassmannian, whose tangent space can be identified with quadratic forms; then $\mathcal C$ is the invariant causal structure modelled on the cone of non-negative-definite quadratic forms, and the resulting notion of positivity on $G$ is the usual one. Returning to the general case,
we define $G^+_{max}$ as the maximal order sandwiched by the Guichardet-Wigner quasimorphism $\mu_G$ as before. Then we have:
\begin{proposition}
With notation as above we have $G^+_{geom} \subset G^+_{max}$\end{proposition}
\begin{proof}
It suffices to show that the Guichardet-Wigner quasimorphism sandwiches $G^+_{geom}$. This is established implicitly in \cite{ClercKoufany} and more explicitly in \cite[Thm. 1.8]{BSH2}.
\end{proof}
\subsection{A continuity property of the Kaneyuki order} The following definition is taken from \cite{Konstantinov}:
\begin{definition} Let $G$ be a topological group, $H$ a closed subgroup and $\preceq$ a partial order on $G/H$. Then $\preceq$ is called \emph{continuous} if
\begin{itemize}
\item[(i)] $[eH, \infty) := \{x \in G/H\,|\, x\geq eH\}$ is closed in $G/H$.
\item[(ii)] The preimage of $[eH, \infty)$ under the quotient map $G \to G/H$ is a locally topologically generated semigroup.
\end{itemize}
\end{definition}
Then we have:
\begin{proposition}\label{KaneyukiCont}
The Kaneyuki order on $\check R$ is continuous.
\end{proposition}
Here, $\check R$ is considered as a homogeneous space under $G$. The proof of Proposition \ref{KaneyukiCont} requires some more structure theory of bounded symmetric domains. We just state the facts we need and refer the reader to \cite{FK, Clerc} for details. Firstly, the bounded symmetric domain $\mathcal D$ associated with $G$ can be realized as the unit ball with respect to the spectral norm in the complexification of a Euclidean Jordan algebra $V$. Denote by $\Omega^0$ the interior of the cone of squares $\Omega$ of $V$. Then $T_\Omega := V +i \Omega^0$ is a tube in $V^\C$ and the Cayley transform associated with $V$ identifies $\mathcal D$ with $T_\Omega$. In particular, $G$ is isomorphic to the universal covering $H := \widetilde{G}(T_\Omega)$ of the group of biholomorphisms $G(T_\Omega)$ of $T_\Omega$. Now denote by $G(\Omega)$ the subgroup of $GL(V)$ preserving $\Omega$, and write $\L g(\Omega)$ for its Lie algebra. Then the Lie algebra $\L g(T_\Omega)$ of $H$ admits a triple decomposition \cite[Sec. 6]{LawsonLim}, i.e. a $\Z$-grading of the form
\[\L g(T_\Omega) = \L g_{-1} \oplus \L g_0 \oplus \L g_1,\]
where $\L g_0 = \L g(\Omega)$ and $\L g_{\pm 1} \cong V$ as vector spaces, and the action of $\L g_0$ on $\L g_{\pm 1}$ is given by the standard action of $\L g(\Omega)$ on $V$, respectively the transpose of this action (with respect to the inner produt making the cone $\Omega$ symmetric). Now the stabilizer of $-e_V$ in $\check S$ is identified via the Cayley transform with the group $P_0 := G(\Omega)G_1$, where $G_1$ is the exponential of $\L g_1$ (see \cite{HartnickDiss}). In particular, $\check S \cong G(T_\Omega)/P_0$. Thus if we denote by $P$ the analytic subgroup of $H$ with Lie algebra $\L g(\Omega) \oplus V$, then $\check R = H/P$. If we furthermore identify
\[T_{eP}\check R \cong \L g_{-1} \cong V,\]
then the causal structure $\mathcal C$ defining the Kaneyuki order is uniquely determined by $\mathcal C_{eP} = \Omega$. With this information we can prove the proposition:
\begin{proof}[Proof of Proposition \ref{KaneyukiCont}] Identify $\check R$ with $H/P$ as above and denote by  $p: H \to H/P$ the canonical projection. Let $W := \Omega \oplus \L g(\Omega) \oplus V$ and $S := \overline{\langle \exp_H(W)\rangle}$. From the above description of the Kaneyuki causal structure and \cite[Thm. 7.2.(iv) and Thm. 7.4.(i)]{Lawson} we deduce firstly, that
\begin{eqnarray}
p^{-1}([eP, \infty)) = S,
\end{eqnarray}
and secondly that $S$ is a Lie semigroup with edge $P$ and Lie wedge $W$. In particular, $p^{-1}([eP, \infty)) = S$ is locally topologically generated. Since the Kaneyuki order is closed, this proves the proposition.
\end{proof}

\subsection{The geometric order and the continuous order} The purpose of this section is to derive the following relation between the geometric and the maximal continuous order on $G$:
\begin{proposition}\label{ContGeom}
With notation as before we have $G^+_{cont} \subset G^+_{geom}$.
\end{proposition}
The main part of the proof is provided by Konstantinov in \cite{Konstantinov}; we only have to combine his results with Proposition \ref{KaneyukiCont}. Indeed we have \cite[Thm. 1]{Konstantinov}:
\begin{theorem}[Konstantinov]\label{Ko} Up to inversion there is a unique continuous order on $\check R$ given by
\begin{eqnarray}\label{Ko2}
x \preceq y \Leftrightarrow \exists g \in G^+_{cont}: gx = y.\end{eqnarray}
\end{theorem}
Now the proposition is immediate:
\begin{proof}[Proof of Proposition \ref{ContGeom}]
By Proposition \ref{KaneyukiCont} and Theorem \ref{Ko} the Kaneyuki order is given by \eqref{Ko2}. Thus for every $g \in G^+_{cont}$ and $x \in \check R$ we have $gx \succeq x$, whence $g \in G^+_{geom}$.
\end{proof}

\def\cprime{$'$}

\bigskip

\textbf{Authors' addresses:}\\

\textsc{Departement Mathematik,  ETH Z\"urich, R\"amistr. 101, 8092 Z\"urich, Switzerland},\\
\texttt{gabi.ben.simon@math.ethz.ch};\\

\textsc{Mathematics Department, Technion - Israel Institute of Technology, Haifa, 3200, Israel},\\
\texttt{hartnick@tx.technion.ac.il}.

\end{document}